\documentclass[11pt]{amsart}
\usepackage{amssymb, amsmath, amsthm}
\usepackage[latin1]{inputenc}
\usepackage{graphicx}
\usepackage[final]{hyperref}
\usepackage[a4paper, centering]{geometry}
\usepackage{color}
\usepackage{graphicx} 

\newtheorem{theorem}{Theorem}
\newtheorem{proposition}[theorem]{Proposition}
\newtheorem{lemma}[theorem]{Lemma}

\theoremstyle{definition}
\newtheorem{definition}[theorem]{Definition}
\theoremstyle{remark}
\newtheorem{remark}[theorem]{Remark}

\def\R{\mathbb{R}}

\def\pscal#1#2{\left\langle#1,\,#2\right\rangle}

\def\osubjet{J^{2, -}_{\Omega}}
\def\osuperjet{J^{2, +}_{\Omega}}
\def\subjet{J^{2, -}_{\overline\Omega}}

\def\csubjet{\overline{J}^{2, -}_{\overline\Omega}}

\def\csubjet{\overline{J}^{2, -}_{\overline\Omega}}

\def\DN{\Delta^N_{\infty}}
\def\Dp{\Delta^+_{\infty}}
\def\Dm{\Delta^-_{\infty}}
\def\Dinf{\Delta_{\infty}}
\def\lmax{\lambda_{\text{max}}}
\def\lmin{\lambda_{\text{min}}}

\DeclareMathOperator{\extr}{extr}

\DeclareMathOperator{\Sym}{\R ^ {n \times n} _{\rm sym}}

\begin{document}

\title[Regularity for the inhomogeneous normalized infinity Laplacian]%
{A  $C^1$ regularity result for the \\ inhomogeneous normalized infinity Laplacian}%
\author[G.~Crasta, I.~Fragal\`a]{Graziano Crasta,  Ilaria Fragal\`a}
\address[Graziano Crasta]{Dipartimento di Matematica ``G.\ Castelnuovo'', Univ.\ di Roma I\\
P.le A.\ Moro 2 -- 00185 Roma (Italy)}
\email{crasta@mat.uniroma1.it}

\address[Ilaria Fragal\`a]{
Dipartimento di Matematica, Politecnico\\
Piazza Leonardo da Vinci, 32 --20133 Milano (Italy)
}
\email{ilaria.fragala@polimi.it}


\date{June 17, 2015}

\begin{abstract} We prove that the unique solution to the Dirichlet problem with constant source term for the 
 inhomogeneous normalized infinity Laplacian on a convex domain of $\R^N$ is of class $C^1$. The result is obtained by showing as an intermediate step the power-concavity (of exponent $1/2$) of the solution. 
\end{abstract}

\maketitle

\section{Introduction}

This paper is focused on the regularity of the unique solution to the Dirichlet problem 
\begin{equation}\label{f:dirich}
	\begin{cases} 
		-\DN u = 1 &\text{in}\ \Omega\\
		u = 0 &\text{on}\ \partial\Omega\,, 
	\end{cases}
\end{equation}

where $\Omega$ is an open bounded convex subset of $\R^N$ and $\DN u$ is the normalized infinity Laplacian. 
The symbolic definition of $\DN  \varphi$ for a smooth function $\varphi$ is
\begin{equation}\label{f:newdef}
\DN \varphi(x) := \begin{cases}
\frac{1}{|\nabla \varphi(x)|^2}\, \pscal{\nabla^2 \varphi (x)\, \nabla \varphi (x)}{\nabla \varphi (x)} & \text{ if } \nabla \varphi (x) \neq 0 
\\
[\lmin (\nabla ^ 2 \varphi(x)), \lmax (\nabla ^ 2 \varphi(x))] & \text{ if } \nabla \varphi (x)=  0\, , 
\end{cases}
\end{equation}
being $\lmin(\nabla^2 \varphi (x))$ and $\lmax(\nabla^2 \varphi (x))$
respectively the minimum and the maximum eigenvalue of the Hessian matrix $\nabla^2 \varphi (x)$;
for the detailed interpretation of the pde $-\DN u = 1$ in the viscosity sense, we refer to Section \ref{secprel} below.

The normalized infinity Laplace operator has recently attracted an increasing interest for its applications and connections with different areas, such as mass transportation \cite{EGan}, shape metamorphism \cite{CEPB}, and especially differential games \cite{KoSe, PSSW}. In fact, according to Kohn and Serfaty \cite{KoSe}, and Peres et al. \cite{PSSW}, the equation $- \DN u = 1$ is satisfied by the continuum value of a differential game called ``tug of war'' (a description can be found  for instance in \cite{ArmSm2012}).

%
%
The development of the existence, uniqueness, and regularity theory for boundary value problems involving the normalized infinity Laplace operator is still at its early stages. 

The well-posedness of the Dirichlet problem \eqref{f:dirich} (possibly with a more general source term $f$) has been proved independently in 
\cite{PSSW} with probabilistic methods,
in \cite{LuWang2} with pde methods (see also \cite{LuWang3}),
and in \cite{ArmSm2012} using a finite difference approach. 
Let us mention that the existence  and uniqueness questions have been attacked also in the parabolic framework, see the interesting paper \cite{JuKa} on the evolution governed by $\DN u$. 

About regularity theory, so far no result seems to be available.  As it is well-known,
the fundamental contributions on regularity  for pde's involving the infinity Laplace operator 
are the celebrated works by Evans-Smart \cite{EvSm} and Evans-Savin \cite{EvSav}, which concern infinity harmonic function and establish they are differentiable in any space dimension and $C ^{1, \alpha}$ in dimension two.  
Recently, some of these results have been extended to a class of inhomogeneous Dirichlet problems for the not-normalized infinity Laplacian  in any dimension:
the everywhere differentiability property  has been obtained by Lindgren  \cite{Lind}, 
and the $C^1$ regularity in case of a constant source term on a convex domain 
has been proved in our previous paper \cite{CFd}. 

In case of the normalized operator, its 
definition via a dichotomy involving the maximum and minimum directions of the Hessian necessarily augments the difficulty of enforcing  regularity techniques, and in fact no result beyond Lipschitz regularity is currently known. 

Aim of this paper is to present two new regularity results for the unique solution to problem~\eqref{f:dirich} on convex domains. 

The first result establishes that the solution is power-concave, precisely, $1/2$-concave, see Theorem \ref{t:concave}. To prove such result, we apply the convex envelope method by Alvarez-Lasry-Lions we already exploited in our previous paper \cite{CFd}, but in the current case this requires a more delicate  procedure
which is outlined for the benefit of the reader at the beginning of Section~\ref{secconc}. In particular, this procedure exploits as a crucial tool  a comparison principle proved in \cite{ArmSm2012}.

The second result is obtained as a consequence of the first one, and states that the solution is of class $C ^1(\Omega)$, see Theorem \ref{t:diff}. 
The proof relies on the local semiconcavity of the solution, combined with an estimate for semiconcave functions near singular points proved in \cite{CFd}. 
 
The paper is organized as follows. In Section \ref{secprel} we provide and discuss the definitions and properties of viscosity solution to  problem \eqref{f:dirich} and to more general second order equations that we need to consider in the proofs. In Sections \ref{secconc} and \ref{secdiffe} we state and prove respectively  the power-concavity and the $C^1$ regularity of the solution.

\section {Preliminaries} \label{secprel}


In this section we review the definition of normalized infinity Laplace operator, and that of viscosity sub- and super-solutions of \eqref{f:dirich}, as well as of more general second order equations (that  we shall need to use). Afterwards, we give some remarks to enlighten some basic features of solutions.

For a $C^2$ function $\varphi$ defined in a neighborhood of $x\in\R^n$,
we define the 
{\it (not normalized) infinity Laplace operator}
\[
\Dinf \varphi(x) := \pscal{\nabla^2\varphi(x)\, \nabla\varphi(x)}{\nabla\varphi(x)}
\]
and the operators
\[
\begin{split}
\Dp\varphi(x) & :=
\begin{cases}
|\nabla\varphi(x)|^{-2}\, \Dinf \varphi(x),
&\text{if}\ \nabla\varphi(x)\neq 0,\\
\lmax(\nabla^2\varphi(x)),
&\text{if}\ \nabla\varphi(x)= 0,
\end{cases}
\\
\Dm\varphi(x) & :=
\begin{cases}
|\nabla\varphi(x)|^{-2}\, \Dinf \varphi(x),
&\text{if}\ \nabla\varphi(x)\neq 0,\\
\lmin(\nabla^2\varphi(x)),
&\text{if}\ \nabla\varphi(x)= 0,
\end{cases}
\end{split}
\]
where, for a symmetric matrix $A\in\Sym(n)$, $\lmin(A)$ and $\lmax(A)$
denote respectively the minimum and the maximum eigenvalue of $A$.

In the following, if $u,v\colon\Omega\to\R$ are two functions and
$x\in\Omega$, by
\[
u \prec_x v
\]
we mean that $u(x) = v(x)$ and $u(y) \leq v(y)$ for every $y\in\Omega$.

Moreover we recall that {\it second order sub-jet} (resp.\ {\it super-jet}),
$\osubjet u (x)$ (resp. $\osuperjet u (x)$), of a function $u\in C(\overline{\Omega})$
at a point $x\in \overline{\Omega}$,  is by definition the set of pairs
$(p, A) \in \R ^n \times \R ^ { n \times n }_{{\rm sym}}$ such that, as $y \to x,\ y\in \overline{\Omega}$, it holds
\begin{equation}\label{d:jet}
 u (y) \geq  (\leq) \ u ( x) + \pscal{ p}{y- x} 
+ \frac{1}{2} \pscal {A (y- x_0)}{y- x_0} + o ( |y - x|^2) 
\,.
\end{equation}

\begin{definition}\label{d:visc}
Let $f\colon\Omega\to\R$ be a continuous function and consider the normalized infinity Laplace equation
\begin{equation}
\label{f:norm}
-\DN u = f(x)\qquad \text{in}\ \Omega\,.
\end{equation}

(i)  An upper semicontinuous function $u\colon\Omega\to\R$
is a {\it viscosity sub-solution of \eqref{f:norm}} if, for every $x\in\Omega$,
\[
-\Dp \varphi(x) \leq f(x)\qquad \forall\varphi\in C^2(\Omega)\ \text{s.t.}\ u\prec_x \varphi.
\]
The explicit formulation reads 
$$
\begin{cases}
-\Dinf \varphi(x) \leq f(x) |\nabla\varphi(x)|^2 & \forall\varphi\in C^2(\Omega)\ \text{s.t.}\ u\prec_x \varphi,\text{ if } \nabla \varphi  (x) \neq 0
\\ 
- \lmax(\nabla^2\varphi(x)) \leq f(x) &  \forall\varphi\in C^2(\Omega)\ \text{s.t.}\ u\prec_x \varphi,\text{ if } \nabla \varphi  (x) =0\,,
\end{cases}
$$
or in terms of super-jets
$$
\begin{cases}
- \langle X p,  p  \rangle \leq f(x) |\nabla\varphi(x)|^2 & \forall(p, X) \in J ^ {2, +} _\Omega u (x), \text{ if } p \neq 0
\\ 
-\lambda _{\max} (X) \leq f(x) & \forall(p, X) \in J ^ {2, +} _\Omega u (x), \text{ if }  p =0\,.
\end{cases}
$$

(ii) A lower semicontinuous function $u\colon\Omega\to\R$
is a {\it viscosity super-solution of~\eqref{f:norm}}
if, for every $x\in\Omega$,
\[
-\Dm \varphi(x) \geq f(x)\qquad \forall\varphi\in C^2(\Omega)\ \text{s.t.}\ \varphi\prec_x u\,.
\]
The explicit formulation reads 
$$
\begin{cases}
-\Dinf \varphi(x) \geq f(x) |\nabla\varphi(x)|^2 & \forall\varphi\in C^2(\Omega)\ \text{s.t.}\ \varphi \prec_x u,\text{ if } \nabla \varphi  (x) \neq 0
\\ 
- \lmin(\nabla^2\varphi(x)) \geq f(x) &  \forall\varphi\in C^2(\Omega)\ \text{s.t.}\ \varphi \prec_x u,\text{ if } \nabla \varphi  (x) =0\,,
\end{cases}
$$
or in terms of super-jets
$$
\begin{cases}
- \langle X p,  p  \rangle \geq f(x) |\nabla\varphi(x)|^2 & \forall(p, X) \in J ^ {2, -} _\Omega u (x), \text{ if } p \neq 0
\\ 
-\lambda _{\min} (X) \geq f(x) & \forall(p, X) \in J ^ {2, -} _\Omega u (x), \text{ if }  p =0\,.
\end{cases}
$$

(iii) A function $u\in C(\Omega)$ is a {\it viscosity solution of~\eqref{f:norm}}
if $u$ is both a viscosity sub-solution and a viscosity super-solution
of~\eqref{f:norm}.

(iv)  A function $u\in C(\overline{\Omega})$ is a {\it viscosity solution
of~\eqref{f:dirich}} if $u=0$ on $\partial\Omega$ and
$u$ is a viscosity solution of~\eqref{f:norm} in $\Omega$.

\end{definition}

\begin{definition}\label{d:general}
Let $I\subset\R$ be an open interval,
 let $H\colon \Omega \times I\times\R^n\times\Sym(n)\to\R$, and consider the equation
\begin{equation}
\label{f:F}
H(x,u, \nabla u, \nabla^2 u) = 0 \qquad \text{in}\ \Omega\,.
\end{equation}
(i) An upper semicontinuous function $u\colon\Omega\to I$
is a {\it viscosity sub-solution of  \eqref{f:F}}
if, for every $x\in\Omega$,
\[
H_*(x, u(x), \nabla\varphi(x), \nabla^2\varphi(x)) \leq 0
\qquad \forall\varphi\in C^2(\Omega)\ \text{s.t.}\ u\prec_x \varphi\, , 
\]
where $H_* $ is the lower semicontinuous envelope of $H$. 

(ii) A lower semicontinuous function $u\colon\Omega\to I$
is a {\it viscosity super-solution of~\eqref{f:F}}
if, for every $x\in\Omega$,
\[
H^*(x, u(x), \nabla\varphi(x), \nabla^2\varphi(x)) \geq 0
\qquad \forall\varphi\in C^2(\Omega)\ \text{s.t.}\ u\prec_x \varphi\,,
\]
where $H^*$ is the upper semicontinuous envelope of $H$. 

(iii) A function $u\in C(\Omega, I)$ is a {\it viscosity solution of~\eqref{f:F}}
if $u$ is both a viscosity sub-solution and a viscosity super-solution
of~\eqref{f:norm}.

(iv) A function $u\in C(\overline{\Omega}, I)$ is a {\it viscosity solution
of the homogeneous Dirichlet problem for equation \eqref{f:F}} if $u=0$ on $\partial\Omega$ and
$u$ is a viscosity solution of~\eqref{f:F} in $\Omega$.

\end{definition}

\begin{remark}
By taking the function
$$H (p, A) := - \Big \langle{A \frac{p}{|p|}}{\frac{p}{|p|}}  \Big \rangle\,, \qquad \forall (p, A) \in (\R^n\setminus \{ 0 \} )\times\Sym(n)\,, $$
and computing its upper and lower semicontinuous envelopes,
we see that Definition  \ref{d:general} gives back Definition  \ref{d:visc}.  (Note that $H _* = - (- H)^*$ and $H ^* = - ( - H ) _*$). 
This argument justifies the apparently  strange
notions of  the operators $\Dp$ and $\Dm$. \end{remark}
\begin{remark}
It is clear that the viscosity solution of the
Dirichlet problem~\eqref{f:dirich} is also a
viscosity solution of
\begin{equation}
\label{f:dirich2}
\begin{cases}
-\Dinf u = |\nabla u|^2 & \text{in}\ \Omega,\\
u = 0 & \text{on}\ \partial\Omega.
\end{cases}
\end{equation}
On the other hand, the converse is not true,
and in fact the Dirichlet problem~\eqref{f:dirich2} has not,
in general, a unique solution. 
To shed some light on this feature, it is enough to look at the one-dimensional  case.  
If $\Omega$ is the interval $(-R, R)$, problems \eqref{f:dirich} and \eqref{f:dirich2}  read respectively
\begin{equation}
\label{f:1dim}
\begin{cases}
- u ''  = 1 & \text{in}\ (-R, R),\\
u (\pm R)= 0 \,.
\end{cases}
\qquad \qquad
\begin{cases}
- u '' (u') ^2 = (u')^2 & \text{in}\ (-R, R),\\
u (\pm R)= 0 \,.
\end{cases}
\end{equation}
It is immediate to check that  the function
$$u _ r (x) := \begin{cases}
\frac{R ^ 2 - r ^ 2}{2} & \text{ if  } |x| \leq r
\\  \noalign{\medskip}
\frac{R ^ 2 - x ^ 2}{2} & \text{ if  }r \leq  |x| \leq R
\end{cases}
$$
is a solution to the second problem in \eqref{f:1dim} for every $r \in [0, R]$, whereas it is a solution to the first problem in~\eqref{f:1dim} only for $r= 0$. 
Indeed, if $r \in (0, R]$,  for every  $x$ with  $|x| < r$ there exist smooth functions $\varphi$ such that $\varphi \prec _ x u$ but  violate the condition $-\varphi '' (x) \geq 1$, so that $u$ is not a super-solution  to $- u '' = 1$. 
\end{remark}

\begin{remark}\label{r:pos}
The viscosity solution $u$ to problem~\eqref{f:dirich} is strictly positive in $\Omega$. 
Indeed, it is nonnegative by the comparison result proved in  
\cite[Thm.~2.18]{ArmSm2012}. 
Assume by contradiction that $u (x_0) = 0$ at some point $x _0 \in \Omega$. Then 
 the function $\varphi \equiv 0$ touches $u$ from below at $x_0$, and hence
$u$ cannot be a viscosity supersolution to the equation 
$- \DN  u =1$ at $x_0$.  
\end{remark}

%

\section{Power-concavity of solutions on convex domains}\label{secconc}

In this section we prove: 
\begin{theorem}\label{t:concave}
Assume that $\Omega$ is an open bounded convex subset of $\R^n$,
and let $u$ be the solution to problem~\eqref{f:dirich}. 
Then $u^{1/2}$ is concave in $\Omega$. 
\end{theorem}
Our proof strategy is the following.
\begin{itemize}
\item[{Step 1.}]
We show that
the map $u \mapsto w := -u^{1/2}$
establishes a one-to-one correspondence between
positive viscosity sub- and super-solution of $-\DN u = 1$ in $\Omega$
and a ``restricted class'' of, respectively, negative viscosity super- and sub-solution of the equation
\begin{equation}
\label{f:modif}
F(w, \nabla w, \nabla^2 w) = 0
\qquad\text{in}\ \Omega,
\end{equation}
where the function $F\colon (-\infty, 0) \times\R^n\times\Sym(n)\to\R$
is defined by
\begin{equation}\label{f:defF}
F(w, p, A) := -\pscal{A p}{p} - \frac{1}{w} \left( |p|^4 + \frac{1}{2} |p|^2\right)\,.
\end{equation}

\item[{Step 2.}]
Under the additional assumption that
\begin{equation}
	\text{the convex set $\Omega$ satisfies an interior sphere condition, }
	\label{f:ipo}
\end{equation}
we can adapt the convex envelope method of
Alvarez, Lasry and Lions (see \cite{ALL}),
proving that
the convex envelope $w_{**}$ of
a (restricted) super-solution $w$ of~\eqref{f:modif}
is still a (restricted) super-solution.

\item[{Step 3.}]
Using the comparison principle
proved in~\cite[Thm.~2.18]{ArmSm2012},
we conclude that, if~\eqref{f:ipo} is fulfilled, then
$w$ is convex, namely $u^{1/2}$ is concave.

\item[{Step 4.}]
By approximating $\Omega$ with outer parallel sets, we finally show that
the assumption~\eqref{f:ipo} can be removed.
\end{itemize}

 \begin{remark}
 The assumption~\eqref{f:ipo} is used in a crucial way to prove
 Lemma~\ref{l:emptyjet} below,
 so that we do not need to impose
 \textsl{state constraints boundary conditions} on $\partial \Omega$
 (see \cite[Definition~2]{ALL}).
 In this respect, we mention that the power concavity of solutions
 to~\eqref{f:modif} has been discussed by Juutinen in \cite{Juu};
 nevertheless, as kindly pointed out by the Author himself,
 his proof is flawed precisely in the argument
 used to show the validity of these
 state constraints boundary conditions.
 (See Lemma~{4.1} in \cite{Juu} where, at boundary points,
 the emptyness
 of $\subjet$ instead of $\csubjet$ is proved.)
 \end{remark}

\begin{remark}
The comparison principle
proved in~\cite[Thm.~2.18]{ArmSm2012}, that we use as a crucial tool in Step 3, holds true for solutions to problem \eqref{f:dirich}. 
In view of the one-to one correspondence mentioned in Step 1, it is therefore irremissible to deal with ``restricted'' solutions to the equation \eqref{f:modif}, according to Definition \ref{d:restricted} below. 
 \end{remark}

\textbf{Step 1.} We set the following
\begin{definition}\label{d:restricted}
We say that
$w$ is a {\it restricted viscosity super-solution of~\eqref{f:modif}}, if it satisfies 
\begin{equation}
\label{f:Rsuper}
\forall x\in\Omega, \ \forall \psi\prec_x w
\ \Longrightarrow \
\begin{cases}
F(w(x), \nabla\psi(x), \nabla^2\psi(x)) \geq 0, 
\\
\lmin(\nabla^2\psi(x)) \leq -\dfrac{1}{2 w(x)}\quad \text{if}\ \nabla \psi (x) = 0.
\end{cases}
\end{equation}
We say that
$w$ is a {\it restricted viscosity sub-solution of~\eqref{f:modif}}, if it satisfies 
\begin{equation}
\label{f:Rsub}
\forall x\in\Omega, \ \forall w\prec_x \psi
\ \Longrightarrow \
\begin{cases}
F(w(x), \nabla\psi(x), \nabla^2\psi(x)) \leq 0,  
\\
\lmax(\nabla^2\psi(x)) \geq -\dfrac{1}{2 w(x)}\quad \text{if}\ \nabla \psi (x) = 0.
\end{cases}
\end{equation}
\end{definition}


\begin{lemma}
\label{l:equiv}
A positive upper semicontinuous function $u\colon\Omega\to\R^+$
is a viscosity sub-solution of $-\DN u = 1$ in $\Omega$
if and only if the function $w = - u^{1/2}$ is a restricted viscosity super-solution of \eqref{f:modif}. 
Likewise, a positive lower semicontinuous function $u\colon\Omega\to\R^+$
is a viscosity super-solution of $-\DN u = 1$ in $\Omega$
if and only if the function $w = - u^{1/2}$ is a restricted viscosity sub-solution of \eqref{f:modif}. 
\end{lemma}

\begin{proof}
We are going to prove only the first part of the statement.
To that aim it is enough to observe that
\[
u \prec_x \varphi
\quad\Longleftrightarrow\quad
\psi := -\varphi^{1/2} \prec_x w,
\]
and the test functions $\varphi$ and $\psi$ satisfy
\[
\nabla\psi(x) = - \frac{\nabla\varphi(x)}{2 \varphi (x) ^ {1/2}}\,,
\qquad
\nabla^2\psi(x) = \frac{1}{4 \varphi(x)^{3/2}}\, \nabla\varphi(x) \otimes \nabla\varphi(x)
- \frac{1}{2\varphi(x)^{1/2}}\, \nabla^2\varphi(x).
\]
In particular, $\nabla\varphi(x) = 0$ if and only if $\nabla\psi(x) = 0$ and,
if this is the case,
$\lmax(\nabla^2\varphi(x)) \geq -1$ if and only if
$\lmin(\nabla^2\psi(x)) \leq -1/(2w(x))$.
\end{proof}

\textbf{Step 2.}
Let us 
show that, if \eqref{f:ipo} is satisfied and $w$ is a restricted viscosity solution to
\begin{equation}
\label{f:probmod}
\begin{cases}
-\Dinf w - \frac{1}{w}\left( |\nabla w|^4 + \frac{1}{2} |\nabla w|^2\right) = 0
& \text{in}\ \Omega,\\
w = 0 & \text{on}\ \partial\Omega,
\end{cases}
\end{equation}
then $w$ is convex. 
We denote by $w_{**}$  the largest convex function below $w$. 
We first establish that, 
under the assumption~\eqref{f:ipo},
for every $x \in \Omega$, in the characterization
\[
w_{**} (x)  = \inf \left\{ 
 \sum _ {i = 1} ^ k \lambda _i w (x_i) \ :\ x = \sum _{i=1} ^k \lambda _ i x _i \, ,\ 
 x_i \in \overline \Omega\, , \ \lambda _i >0\, ,\ \sum  _ {i = 1} ^ k \lambda _i = 1\, , \ k \leq n+1 \right\}\,
\] 
the infimum can be attained only at interior points $x_i \in \Omega$:

\begin{lemma}\label{l:emptyjet}
Assume \eqref{f:ipo}, and let $u$ be the  solution to problem~\eqref{f:dirich}. 
Set $w := - u^{1/2}$.
For a fixed $x\in\Omega$, let $x_1,\ldots,x_k\in\overline{\Omega}$,
$\lambda_1,\ldots,\lambda_k > 0$, with  $\sum_{i=1} ^ k \lambda _ i = 1$, be such that
\[
x= \sum_{i=1}^k \lambda_i x_i\, , \quad
w_{**}(x) = \sum_{i=1}^k \lambda_i w(x_i).
\] 
Then $x_1, \ldots, x_k\in\Omega$.
\end{lemma}

\begin{proof}
Assume by contradiction that at least one of the $x_i$'s, say $x_1$,
belongs to $\partial\Omega$.
Let $B _ R (y)\subset \Omega$ be a ball such that $\partial B _ R (y) \cap \partial \Omega = \{ x_1 \}$.
Since $-\Delta_\infty u = |\nabla u|^2$, by Lemma~2.2 in \cite{CEG} the function $\tilde u := -u$ enjoys the property of comparison with cones from above according to Definition 2.3 in the same paper. 
Then, by Lemma 2.4 in \cite{CEG}, the function
\[
r \mapsto \max _{x \in \partial B _ r (y)} \frac{\tilde u (x) -  \tilde u  (y)}{r} = - \min  _{x \in \partial B _ r (y)}  \frac{ u (x) -  u  (y)}{r}
\]
is monotone nondecreasing on the interval $(0, R)$. 
Namely, for all $r \in (0, R)$, there holds
\begin{equation}
\label{f:cones2}
\min_{x \in \partial B _ r (y)} \frac{ u (x) -  u  (y)}{|x-y|}  \geq \min_{x \in \partial B _ R (y)} \frac{ u (x) -  u  (y)}{|x-y|} = -  \frac{u (y)} {R}\ ,
\end{equation}
where the last equality comes from the fact that
$u$ is non-negative in $\Omega$ (cf.\ Remark \ref{r:pos}).  
By \eqref{f:cones2}, we have
\[
u (x) \geq u (y ) \Big ( 1 - \frac{|x-y|}{R} \Big ) \qquad \forall x \in B _R (y) \, ,
\]
and hence
\begin{equation}
\label{f:upbd}
w (x) \leq w (y ) \Big ( 1 - \frac{|x-y|}{R} \Big ) ^ {1/2} \qquad \forall x \in B _R (y) \, .
\end{equation}
Let us define the unit vector $\zeta := (x-x_1) / |x-x_1|$
and let $\nu = (y-x_1) / |y-x_1|$ denote the inner normal of 
$\partial\Omega$ at $x_1$.
Since $\Omega$ is a convex set and $x\in\Omega$, we have
that $\pscal{\zeta}{\nu} > 0$
and $x_1 + t\zeta\in B_R(y)$ for $t>0$ small enough.
Moreover, $w_{**}$ is affine on $[x_1, x]$: indeed, since the epigraph of $w_{**}$ is the convex envelope of the epigraph of $w$, it is readily seen that $w_{**}$ is affine on the whole set of convex combinations of the points $\{x_1, \dots, x_k \}$. 
Taking into account that $w_{**}(x_1) = w(x_1) = 0$, we infer that
there exists $\mu > 0$ such that
\[
w(x_1 + t\zeta) \geq w_{**}(x_1+t \zeta) = - \mu t
\qquad\forall t\in [0,1].
\] 
{}From \eqref{f:upbd} we obtain
\[
-\mu t \leq w(y) \left( 1 - \frac{|t\zeta - R\nu|}{R}\right)^{1/2}
= w(y) \left(\pscal{\zeta}{\nu} \frac{t}{R} + o(t)\right)^{1/2},
\qquad t\to 0^+,
\]
and, recalling that $w(y) < 0$, 
\[
\mu t^{1/2} \geq K + o(1),\qquad t\to 0^+
\]
with $K>0$, a contradiction.
\end{proof}

On the basis of the lemma just proved, we obtain:

\begin{lemma}\label{p:All}
Assume \eqref{f:ipo}, and let $w$ be a restricted viscosity super-solution to~\eqref{f:probmod}.  
Then also $w_{**}$ is a restricted viscosity 
super-solution to the same problem.
\end{lemma}

\begin{proof}

Let $w$ be a restricted viscosity super-solution to~\eqref{f:probmod}.
In order to show that $w _{**}$ is still a restricted viscosity super-solution to the same problem, we begin by observing that that $w _{**}$ agrees with $w$ on $\partial \Omega$, namely $w_{**} = 0$ on $\partial \Omega$  (since \cite[Lemma 4.1]{ALL} applies). 

Now let us check that $w_{**}$ satisfies ~\eqref{f:Rsuper}. In terms 
of sub-jets, such property can be rephrased as 
\begin{equation}
\label{f:Rsuper2}
\forall x\in\Omega,\
\forall (p,A)\in\subjet w_{**}(x)
\ \Longrightarrow\
\begin{cases}
F(w_{**}(x), p, A) \geq 0  \quad \text{if}\ p \neq 0,\\
\lmin(A) \leq -\dfrac{1}{2 w_{**}(x)} \quad \text{if}\ p = 0.
\end{cases}
\end{equation}

Let $x\in\Omega$  and consider first the case when $(p,A)\in\subjet w_{**}(x)$, with $p \neq 0$.

For every $\epsilon>0$ small enough,
applying Proposition~1 in \cite{ALL}
and Lemma~\ref{l:emptyjet},
we obtain
points $x_1,\ldots, x_k \in \Omega$,
positive numbers $\lambda_1,\ldots,\lambda_k$ satisfying
$\sum_{i=1}^k \lambda_i = 1$,
and elements
$(p, A_i) \in \csubjet w(x_i)$, with 
$A_i$ positive semidefinite,
such that
\[
\sum_{i=1}^k \lambda_i x_i = x,\quad
\sum_{i=1}^k \lambda_i w(x_i) = w_{**}(x),\quad
A-\epsilon A^2 \leq 
\left(\sum_{i=1}^k \lambda_i A_i^{-1}\right)^{-1} =: B.
\]
We recall that, here and in the sequel,
it is not restrictive to assume that 
the matrices $A$, $A_1,\ldots, A_k$ are positive definite, 
since the case of degenerate matrices can be handled
as in \cite{ALL}, p.~273.
Moreover we recall that the ``closure'' $\subjet v (x_0)$
is the set of $(p,A)\in \R ^n \times \R ^ { n \times n }_{{\rm sym}}$
for which there is a sequence $(p_j, A_j)\in \subjet v(x_j)$ (according to \eqref{d:jet})
such that
$(x_j, v(x_j), p_j, A_j) \to (x_0, v(x_0), p, A)$. 

Then, since by assumption $w$ is a super-solution to~\eqref{f:probmod}, we have $F(w(x_i), p, A_i) \geq 0$, i.e.
\[
-w(x_i) \leq \frac{1}{\pscal{A_i p}{p}}\left(|p|^4 + \frac{1}{2} |p|^2 \right)\, , 
\]
so that
\[
- \frac{1}{\sum_{i=1}^k \lambda_i w(x_i)}
\left(|p|^4 + \frac{1}{2} |p|^2 \right)
\geq
\left(\sum_{i=1}^k \lambda_i \frac{1}{\pscal{A_i p}{p}}\right)^{-1}.
\]

\smallskip
Then, using the degenerate ellipticity of $F$ and the concavity of the map
$Q\mapsto 1/ \text{tr} \big ( (p \otimes p )  Q^{-1} \big )$ 
(see \cite{ALL}, p. 286), we obtain
\[
\begin{split}
F(w_{**}(x), p, A-\epsilon A^2) & \geq
-\pscal{B p}{p}
- \frac{1}{\sum_{i=1}^k \lambda_i w(x_i)}
\left(|p|^4 + \frac{1}{2} |p|^2 \right)
\\ & \geq
-\pscal{B p}{p}
+ \left(\sum_{i=1}^k \lambda_i \frac{1}{\pscal{A_i p}{p}}\right)^{-1}
\\ & \geq 0\,.
\end{split}
\]

\medskip
It remains to consider the case when $(0,A)\in\subjet w_{**}(x)$. We have to show that 
$$\lmin(A) \leq -\frac{1}{2 w_{**}(x)} \,.$$
In terms of test functions, this is equivalent to prove that
\[
\psi \prec_x w_{**},\
\nabla\psi(x) = 0
\quad\Longrightarrow\quad
\lmin(\nabla^2\psi(x)) \leq - \frac{1}{2 w_{**}(x)}\,.
\]

Since $w_{**}$ is a convex function,
the conditions
$\psi \prec_x w_{**}$ and $\nabla\psi(x) = 0$
imply that $x$ is a minimum point of $w_{**}$.
In particular, we must have
$w(x_1) = \cdots = w(x_k) = w_{**}(x)$. 

If $k=1$, then $w_{**}(x) = w(x)$, $B = A_1$ and
$\lmin(A-\epsilon A^2)\leq \lmin(B) = \lmin(A_1) \leq -1/(2 w(x))$,
so that the required inequality follows.

Assume now that $k>1$, 
so that $x$ is not a strict minimum point of $w_{**}$.
Since $x$ belongs to the relative interior of the convex polyhedron with
vertices $x_1, \ldots, x_k$, if we choose
$q := (x_1 - x) / |x_1 - x|$ we get that
$w_{**}(x + t q)$ is constant for $|t|$ small enough,
so that $\psi(x + tq) \leq w_{**}(x) = \psi(x)$ for $|t|$ small.
Hence
\[
\lmin(\nabla^2\psi(x)) \leq \pscal{\nabla^2\psi(x)\, q}{q} \leq 0
< - \frac{1}{2 w_{**}(x)}
\]
completing the proof.
\end{proof}

\medskip
\noindent
\textbf{Step 3.}
Let us prove that, under the additional assumption~\eqref{f:ipo},
the unique solution $u$ to~\eqref{f:dirich}
is $1/2$-power concave.


Let  $w= - u^{1/2}$.
By Lemma~\ref{l:equiv}, $w$ is a restricted super-solution to~\eqref{f:probmod}
hence, by Lemma~\ref{p:All}, 
also  $w_{**}$ is a restricted super-solution to~\eqref{f:probmod}.
Invoking again Lemma~\ref{l:equiv},
the function $v := (w_{**})^2$ is a viscosity
sub-solution to~\eqref{f:dirich}.
By the comparison principle proved in~\cite[Thm.~2.18]{ArmSm2012}
we deduce that $v \leq u$,
i.e.\ $(w_{**})^2 \leq w^2$.
On the other hand,
since $w_{**}\leq w$ (by definition of convex envelope) and $w\leq 0$,
we have that
$(w_{**})^2 \geq w^2$,
so that $w = w_{**}$, namely $w$ is a convex function.

\medskip
\noindent
\textbf{Step 4.}
Let us finally show that the conclusions of Step 3 (i.e.\ the power concavity of $u$)
remains true if $\Omega$ is any bounded convex domain.

For $\varepsilon\in (0,1]$ let $\Omega_\varepsilon$ denote
the outer parallel body of $\Omega$ defined by
\[
\Omega_\varepsilon := \{x\in\R^n:\ \text{dist}(x, \Omega) < \varepsilon\}\,,
\]
and let $u_\varepsilon$ denote the solution to
\[
	\begin{cases} 
	-\DN u_\varepsilon = 1 &\text{in}\ \Omega_\varepsilon\\
	u_\varepsilon = 0 &\text{on}\ \partial\Omega_\varepsilon\,. 
	\end{cases}
\]
Since $\Omega_\varepsilon$ satisfies an interior sphere condition
(of radius $\varepsilon$), by Step~2 the function
$u_\varepsilon^{1/2}$ is concave in $\Omega_\varepsilon$.
Therefore, to show that $u^{1/2}$ is concave in $\overline{\Omega}$,
it is enough to show that, as $\varepsilon \to 0$, $u_\varepsilon \to u$
uniformly in $\overline{\Omega}$.
In turn, by Theorem~5.3 in \cite{LuWang2}, this convergence
holds true provided that $\left.{u_\varepsilon}\right|_{\partial\Omega}$
converges uniformly to $0$. 

Let $y\in\partial\Omega$, and let $x_\varepsilon\in\partial\Omega_\varepsilon$
be such that $|x_\varepsilon - y| = \varepsilon$.
Let us consider the polar quadratic polynomial
\[
\eta(x) := \frac{1}{2}\, \text{diam} (\Omega_\varepsilon) \, |x-x_\varepsilon|
- \frac{1}{2}\, |x-x_\varepsilon|^2.
\]
Since $u_\varepsilon \leq \eta$ on $\partial\Omega_\varepsilon$,
and $u_\varepsilon$ enjoys the comparison with quadratic cones
(see \cite[Theorem~2.2]{LuWang2} or \cite[Lemma~5.1]{ArmSm2012}),
we have $u_\varepsilon \leq \eta$ on $\Omega_\varepsilon$,
and, in particular,
\[
u_\varepsilon(y) \leq \frac{\varepsilon}{2}\, (\text{diam} (\Omega) + 1) 
- \frac{1}{2}\, \varepsilon^2.
\] 
Hence  $\left.{u_\varepsilon}\right|_{\partial\Omega}$
converges uniformly to $0$ in $\partial\Omega_\varepsilon$.


\section{Local semiconcavity and $C^1$-regularity  of solutions}\label{secdiffe}

In this section we show the local semiconcavity and the $C ^1$-regularity of the unique solution to problem \eqref{f:dirich}. 

We recall that  $u\colon \Omega \to \R$ is called {\it semiconcave (with constant $C$) in $\Omega$} if 
\[
u(\lambda x + (1-\lambda)y) \geq \lambda u(x) + (1-\lambda) u(y)
-C\frac{\lambda(1-\lambda)}{2}\, |x-y|^2
\qquad \forall [x,y] \subset \Omega \ \text { and } \ \forall \lambda \in [0,1]\,.
\]
We say that $u$ is {\it locally semiconcave in $\Omega$} if it is semiconcave on compact subsets of $\Omega$. 

\begin{proposition}\label{c:locsemiconc}
Assume that $\Omega$ is an open bounded convex subset of $\R^n$,
and let $u$ be the solution to problem~\eqref{f:dirich}. 
Then $u$ is locally semiconcave in $\Omega$.
\end{proposition}

\begin{proof} 
Let $K$ be a compact convex subset of $\Omega$,
and let $M$ be the Lipschitz constant of $v:= u ^ {1/2}$ in $K$.
We claim that $u$ is semiconcave in $K$ with semiconcavity constant
$C = 2 M^2$.
Namely, given $x,y\in K$ and $\lambda\in [0,1]$
and using the concavity of $v$ established in Theorem \ref{t:concave},  we get
\[
\begin{split}
& u(\lambda x + (1-\lambda)y) - \lambda u(x) - (1-\lambda) u(y) + \frac{C}{2}\lambda(1-\lambda)|x-y|^2
\\ & \qquad \geq [\lambda v(x) + (1-\lambda) v(y)]^2 - \lambda v(x)^2 - (1-\lambda) v(y)^2 + 
M^2 \lambda(1-\lambda)|x-y|^2
\\ & \qquad
= \lambda(1-\lambda) \left[M^2 |x-y|^2 - |v(x) - v(y)|^2\right]
\geq 0\,.
\qedhere
\end{split}
\]
\end{proof}
\begin{remark}
In the above result  the semiconcavity property of $u$  is stated  just {\it locally} in $\Omega$. The reason can understood by inspection of the proof, and analyzing the behaviour of the constant $M$ appearing therein as the compact set $K \uparrow \Omega$. Indeed, recalling that the function $w= -v = -u ^ {1/2}$ satisfies  \eqref{f:upbd},
choosing $x = x_1 + \lambda \nu$, 
and taking into account that $w (x_1) = 0$,
it is readily seen that
\[
\lim _{\lambda \to 0 ^ +} 
\frac{w ( x_1 + \lambda \nu) - w (x_1)}{ \lambda} 
\leq  \lim _{\lambda \to 0 ^ +} \frac{w ( y)}{ \lambda} \Big ( \frac{\lambda}{R} \Big ) ^ {1/2} = - \infty\,.
\]
This means that the normal derivative of $w$ with respect to
the external normal is $+\infty$ at every boundary point of $\Omega$ (so that $M \to + \infty$ as $K \uparrow \Omega$).

\end{remark}

Next,   let us quote an estimate for locally semiconcave functions near singular points
that we are going to exploit in order to arrive at the $C ^1$-regularity; such result 
was proved in our previous paper~\cite[Thm.~8]{CFd} (see also~ \cite[Thm.~5]{CFc}).

Given a function  $u\in C(\Omega)$, we denote by $\Sigma (u)$ the singular set of $u$, namely the set of points where $u$ is not differentiable. 
For every $x_0\in\Sigma(u)$,
 the {\it super-differential} of 
$u$ at $x_0$, which is defined by
\[
D^+ u (x_0) := 
\left\{ p \in \R ^n \ :\ \limsup _{ x \to x_0} 
\frac{u(x) - u(x_0) -\langle p, x-x_0 \rangle }{|x-x_0| }  
\leq 0  \right\} \,,
\]
turns out to be a  
nonempty compact convex set different from a singleton. In particular,  $D^+ u(x_0)\setminus \extr D^+ u(x_0)$
is not empty and contains non-zero elements.

\begin{theorem}\label{t:estid} {\rm (\cite[Theorem 8]{CFd})}
Let $u\colon\Omega\to\R$ be a locally semiconcave function,
let $x_0 \in \Sigma(u)$, and let $p\in D^+ u(x_0)\setminus \extr D^+ u(x_0)$. 
Let 
$R>0$ be such that
$\overline{B}_R(x_0)\subset\Omega$, and let $C$ denote
the semiconcavity constant of $u$ on 
$\overline{B}_R(x_0)$.
Then  
there exist
a constant $K>0$ and a unit vector
$\zeta\in\R^n$ 
satisfying the following property:
\begin{equation}\label{f:estidist0}
u(x) \leq u(x_0) + \pscal{p}{x-x_0}
- K\, |\pscal{\zeta}{x-x_0}| +\frac{C}{2}
|x-x_0|^2
\quad \forall x\in\overline{B}_R(x_0)\,.
\end{equation}
In particular, for every $c>0$, setting  $\delta := \min\{K/c, R\}$, 
it holds
\begin{equation}\label{f:estidist}
u(x) \leq u(x_0) + \pscal{p}{x-x_0}
- c \pscal{\zeta}{x-x_0}^2 +\frac{C}{2}
|x-x_0|^2 \quad  \forall x \in  \overline B_{\delta}(x_0)\,.
\end{equation}
Furthermore, if $p\neq 0$ then the vector $\zeta$ can
be chosen so that $\pscal{\zeta}{p}\neq 0$.
\end{theorem}


\bigskip
We are now ready to give our  $C^1$-regularity result:

\begin{theorem}\label{t:diff}
Let $u\in C(\Omega)$ be a viscosity solution to
$-\Delta_{\infty}u = f(x,u)$ in $\Omega$.
If $u$ is locally semiconcave in $\Omega$,
then $u$ is everywhere differentiable (hence of class $C^1$)
in $\Omega$.

In particular, if $\Omega$ is an open bounded convex subset of $\R^n$,
and  $u$ is the unique solution to problem~\eqref{f:dirich}, then $u \in C ^ 1 (\Omega)$. 
\end{theorem}

\begin{proof} Let $u\in C(\Omega)$ be a locally semiconcave viscosity solution to
$-\Delta_{\infty}u = f(x,u)$ in $\Omega$.
Assume by contradiction that $\Sigma(u)\neq \emptyset$.
Without loss of generality we can assume that $0\in \Sigma(u)$.
Let $p\in D^+u(0)\setminus \extr D^+u(0)$, $p\neq 0$.
By Theorem~\ref{t:estid}, there exists a unit vector 
$\zeta\in\R^n$, with $\pscal{\zeta}{p}\neq 0$, such that,  for every $c>0$, the inequality
\[
u(x) \leq 
u(0) + \pscal{p}{x} - c \pscal{\zeta}{x}^2 +
\frac{C}{2}\, |x|^2
\]
holds true for all $x\in B_{\delta}(0)$, with $\delta$ depending on $c$.
Thus, setting $\varphi (x):= u(0) + \pscal{p}{x} - c \pscal{\zeta}{x}^2 +
\frac{C}{2}\, |x|^2$, it holds $u \prec_0\varphi $.
Since $\nabla \varphi (0) = p \neq 0$, we have
\[
\Dp \varphi(0) =
-2 c  \pscal{\zeta}{p}^2 + C|p|^2,
\]
choosing $c>0$ large enough we get $-\Dp \varphi(0) > f(0, u(0))$,
a contradiction.

Since $u$ is differentiable everywhere in $\Omega$,
by \cite[Prop.~3.3.4]{CaSi}
we conclude that $u\in C^1(\Omega)$.

Finally, if $\Omega$ is an open bounded convex subset of $\R^n$,
and  $u$ is the unique solution to problem~\eqref{f:dirich}, the first part of the statement just proved applies (and hence $u \in C ^ 1 (\Omega)$) because we know from Proposition \ref{c:locsemiconc} that $u$ is locally semiconcave. 

\end{proof}


\def\cprime{$'$}
\providecommand{\bysame}{\leavevmode\hbox to3em{\hrulefill}\thinspace}
\providecommand{\MR}{\relax\ifhmode\unskip\space\fi MR }
\providecommand{\MRhref}[2]{%
	\href{http://www.ams.org/mathscinet-getitem?mr=#1}{#2}
}
\providecommand{\href}[2]{#2}

\end{document}